\documentclass[12pt]{article}
\usepackage[margin=1in]{geometry}
\usepackage{lmodern}
\usepackage[T1]{fontenc}
\usepackage{amssymb,amsthm,amsmath,amsopn}
\usepackage[normal]{subfigure}
\usepackage{import}
\usepackage{verbatim}
\usepackage{mathrsfs}
\usepackage{colonequals}
\usepackage{paralist}
\usepackage{ctable}
\usepackage{bm}
\usepackage{mathtools}
\usepackage[colorlinks=true,linkcolor=blue,citecolor=blue]{hyperref}%
\usepackage[capitalise,noabbrev]{cleveref}
\usepackage{setspace}
\usepackage{microtype}
\usepackage{caption}

\graphicspath{{./images/}}

\newtheorem{theorem}{Theorem}[section]

\newtheorem{lemma}[theorem]{Lemma}

\theoremstyle{definition}

\newcommand{\R}{\mathbb{R}}

\newcommand{\A}{\mathcal{A}}

\newcommand{\be}{\begin{equation}}
\newcommand{\ee}{\end{equation}}
\newcommand{\bea}{\begin{align}}
\newcommand{\eea}{\end{align}}
\newcommand{\bead}{\begin{aligned}}
\newcommand{\eead}{\end{aligned}}
\newcommand{\bsub}{\begin{subequations}}
\newcommand{\esub}{\end{subequations}}
\newcommand{\bgd}{\begin{gathered}}
\newcommand{\egd}{\end{gathered}}
\newcommand{\vect}[1]{\bm{#1}}

\newcommand{\ce}{\colonequals}

\newcommand{\til}[1]{\tilde{#1}}

\newcommand{\tran}{^T} 

\DeclareMathOperator{\trace}{tr}
\DeclareMathOperator{\V}{vec}
\DeclareMathOperator{\diag}{diag}

\def\grad{\nabla} 
\def\div{\tsp\mathrm{div}}

\newcommand{\tsp}{\;\!}
\newcommand{\stnsp}{\!\!\:}
\newcommand{\ssp}[1]{{#1 \stnsp}}

\newcommand*{\diff}{\mathop{}\!\mathrm{d}}

\newcommand{\vphi}{\varphi}
\newcommand{\lam}{\lambda}

\newcommand{\Lam}{\Lambda}

\newcommand{\bx}{\bm{x}}
\newcommand{\n}{\bm{n}}

\newcommand{\ip}[2]{\langle#1,#2\rangle}

\newcommand{\vx}{\vect{x}}
\newcommand{\vy}{\vect{y}}
\newcommand{\vz}{\vect{z}}
\newcommand{\vn}{\vect{n}}

\usepackage{titlesec}

\newcommand{\titleauthor}[2]{
	\hphantom{.}
	\vspace{.5in}
	\begin{center}
		\textbf{\uppercase{#1}}
	\end{center}
	\begin{center}
		{\small \uppercase{#2}}
	\end{center}}
	
\newcommand{\ab}[1]{
	\vspace{0in}
		{\footnotesize \textbf{Abstract.}~{#1}}
	\vspace{0in}}
	
\newcommand{\kw}[1]{
	\vspace{0in}
	\noindent\quad{\footnotesize \textbf{Keywords.}~{#1}}\\
	\vspace{0in}}
	
\newcommand{\amssub}[1]{
	\vspace{0in}
	\noindent\quad{\footnotesize \textbf{AMS subject classifications.}~{#1}}\\
	\vspace{0in}}

\titleformat{\title}[hang]{\large\bfseries}{}{0em}{\LARGE\bfseries}[]
\titleformat{\section}[hang]{\bf}{\thesection.}{.5em}{}[]
\titleformat{\subsection}[hang]{\it}{\thesubsection.}{.5em}{}[]
\titleformat{\subsection}[hang]{\it}{\thesubsection.}{.5em}{}[]
\titleformat{\subsubsection}[hang]{\it}{\thesubsubsection.}{.5em}{}[]

\begin{document}

\titleauthor{A Continuation Method for Computing Constant Mean Curvature Surfaces with Boundary}{N.\,D. Brubaker\footnotemark[1]}

\footnotetext[1]{Department of Mathematics and Center for Computational and Applied Mathematics, California State University, Fullerton, Fullerton, CA, 92831 (\texttt{nbrubaker@fullerton.edu})}

\ab{Defined mathematically as critical points of surface area subject to a volume constraint, constant mean curvatures (CMC) surfaces are idealizations of interfaces occurring between two immiscible fluids. Their behavior elucidates phenomena seen in many microscale systems of applied science and engineering; however, explicitly computing the shapes of CMC surfaces is often impossible, especially when the boundary of the interface is fixed and parameters, such as the volume enclosed by the surface, vary. In this work, we propose a novel method for computing discrete versions of CMC surfaces based on solving a quasilinear, elliptic partial differential equation that is derived from writing the unknown surface as a normal graph over another known CMC surface. The partial differential equation is then solved using an arc-length continuation algorithm, and the resulting algorithm produces a continuous family of CMC surfaces for varying volume whose physical stability is known. In addition to providing details of the algorithm, various test examples are presented to highlight the efficacy, accuracy and robustness of the proposed approach.}\\

\kw{constant mean curvature, interface, capillary surface, symmetry-breaking bifurcation, arc-length continuation}

\amssub{49K20, 53A05, 53A10, 76B45}

\section{Introduction}

Determining the behavior of an interface between nonmixing phases is crucial for understanding the onset of phenomena in many microscale systems. For example, interfaces induce capillary action in tubules~\cite{finn2002eight}, produce beading in microfluidics~\cite{lipowsky1999liquid}, and change the wetting properties of patterned substrates~\cite{lau2003super}. Additionally, stiction, the leading cause of failure in manufacturing microelectromechanical systems devices, is caused by an interfacial tension~\cite{tas1996stiction}. 

Constant mean curvature (CMC) surfaces are the simplest idealization of an interface and are mathematically defined as critical points of the surface area functional subject to a volume constraint. Physically, the surface area functional represents an energy arising from a differential attraction at the interface, while the volume constraint captures incompressibility of the fluid~\cite{finn1986equilibrium}. A necessary and sufficient condition for such critical points is that 
\be\label{cmc_eq}
 2H = \lam
\ee
at every point on the surface, where $H$ denotes the mean of the surface's two principle curvatures and the parameter $\lam$, whose value gives the magnitude of the hydrostatic pressure jump, is a Lagrange multiplier used for setting the volume encapsulated by the surface. In applications, interfaces are commonly attached to rigid components, such as substrates or rods~(cf.~\cite{bostwick2015stab}), so that the CMC equations must be coupled with boundary conditions that fix the boundary of the closed curve. The resulting system of equations is called Plateau's problem.

Equation \eqref{cmc_eq} is highly nonlinear, so predicting the configuration of an interface with prescribed boundary is difficult. Multiple existence theorems have been proved, along with nonuniqueness due to the presence of at least two geometrically distinct ``small'' and ``large'' solutions, yet not much else is known about the solution set; see~\cite{lopez2013cmc}. Even when the boundary $\Gamma$ is a circle, it has essentially only been determined that (1) $|H|$ is necessarily less than the inverse of the radius of $\Gamma$~\cite{ehinz1969nonexist}; (2) there are two spherical cap solutions; and (3) there is a nonrotational, self-intersecting, compact CMC surface with genus greater than 2~\cite{kapouleas1991compact}.

With these analytical challenges, significant effort has been invested in developing numerical methods to construct discrete analogs of CMC surfaces. The most common approach is to triangulate the surface and then approximate geometric quantities on the mesh to construct an optimization problems for determining the locations of the vertices; see~\cite{polthier2002discrete,metzger2004numerical,dzuik2006fem,pan2012robust,crane2013robust,renka2015simple}. Notably, Surface Evolver program~\cite{brakke1992surface}, basded on this approach, has been used to solve numerous engineering problems in, for example, soldering~\cite{singler1996computer}, capillarity~\cite{collicott2004computing,peraud2014geometry}, and mechanics~\cite{paulsen2015optimal}. A second common approach, built upon an analytic construction in~\cite{wente1986counterexample}, uses quadrilateral nets to approximate the surface, after which a corresponding discrete integrable system is formed and solved; see~\cite{bobenko1999disc} for isothermal constructions and \cite{dorfmeister1998weierstrass} for the Dorfmeister, Pedit, and Wu recipe. 

Both of the above methods have limiting factors when being used to solve Plateau problems arising in applications. Integrable systems methods are mostly used to construct complete surfaces and cannot be easily adapted to preserve boundary curves. Triangulation methods, which have no trouble fixing a boundary, find energy minimizers via direct approaches and, consequently, can only find stable CMC surfaces, i.e., local minimizers. Additionally, seemingly all of the discrete approaches are not well adapted to find families of CMC surfaces when the prescribed curvature is smoothly varied, a process that is often necessary in applications and is analogous to varying the strength of the pressure induced by hydrostatics. 

In this paper we propose a new method for numerically finding CMC surfaces with fixed boundary that is  different from both the integrable systems and triangulation approaches discussed. In particular, we treat \eqref{cmc_eq} as a classical partial differential equation, over some fixed domain $\Omega \subset \R^2$, that determines the parameterization $\bx \colon \Omega \to \R^3$ of the surface. The resulting partial differential equation is then discretized and solved. Of course, this process takes some care because \eqref{cmc_eq}, in terms of unknown coordinates $(x,y,z)$ of the parameterization $\bx$, is underdetermined. Usually this indeterminacy is removed by rewriting the surface $\bx$ as a graph%
\footnote{It is also tempting is to use the differential geometric identity $\Delta \bx = 2 H \n$ for the surface's intrinsic Laplacian $\Delta$ and normal vector $\n$, which gives a quasilinear partial differential equation for each of the coordinates~\cite{finn1992greens}. Unfortunately, conformal invariances of the surface make these equations dependent, and additional auxiliary conditions must be added, which can make the resulting system overdetermined and increases the dimension of the problem; see~\cite{myshkis1987lowgrav}.} 
$z = z(x,y)$ over $\Omega$ that spans the curve $\Gamma$, and \eqref{cmc_eq} reduces to a nonparametric, quasilinear partial differential equation with Dirichlet data~(cf.~\cite{lobaton2007computation}); however, a graph of this form can only construct surfaces that simply project on hyperplanes of $\R^3$, removing our ability to capture large cap solutions. Motivated by this procedure, we will look for CMC surfaces represented as a normal graph $\bx = \bx_0 + \vphi \tsp\n_0$, where the function $\vphi\colon \Omega \to \R$ is to be determined and the map $\bx_0:\Omega \to \R^3$ is a known parameterization of another CMC surface sufficiently close to $\bx$ whose normal vector is $\n_0$. Theory of such a representations is well established in the literature~\cite{koiso2002deformation}. 

Our proposed method utilizes numerical arc-length continuation to determine a family of CMC surfaces. Arc-length continuation is a common and efficient method for solving general equations of the form $f(\vphi,\lam) = 0$ for a mapping $f \colon X \times \R \to Y$, where $X \subset \R^k$ and $Y \subset \R^{k}$ are the discrete approximates of given function spaces~\cite{allgower1990numerical,doedel1991numerical,doedel2007auto,keller1987lectures}. Here, $\vphi \in X$ is the desired solution, and $\lam$ is a free bifurcation parameter, also known as a nonlinear eigenvalue. Given an initial solution $(\vphi_0,\lam_0)$ of $f = 0$, these methods robustly trace out the remainder of the connected solution branch by finding of sequence of solutions along the curve $\gamma(s) = (\vphi(s),\lam(s))$, which is parameterized as a function of the branch's arc length measured relative to the initial point. The new solutions $\gamma_{i+1}$ are determined from a given solution $\gamma_i$ in two steps. First, an Euler predictor step of specified size%
\footnote{The sign of $h$ determines which direction the curve is traversed.} 
$h$ is taken in a direction tangent to the curve to produce a guess $\til{\gamma}_{i+1}$, i.e.,
\[
 \til{\gamma}_{i+1} = \gamma_i + h \, t(f'(\gamma_i)),
\]
where $t(f'(\gamma_i))$ is the unique normalized vector in the null-space of the Jacobian $J \in \R^{k\times (k+1)}$ of $f$ at regular points of the corresponding mapping. Then the guess $\til{\gamma}_{i+1}$ is corrected to a new solution $\gamma_{i+1}$ by looking for the point on the curve $(x(s),\lam(s))$ nearest to $\til{\gamma}_{i+1}$ via the optimization problem
\[
 \min_{\gamma \in X \times \R }\{\|\gamma - \til{\gamma}_{i+1}\| : f(\gamma) = 0\}.
\]
The corresponding solution can be found with a Newton-like iteration $\eta^{(k+1)} = \eta^{(k)} - (f')^+(\eta^{(k)}) \, f'(\eta^{(k)})$ with starting condition $\eta^{(0)} = \til{\gamma}_{i+1}$, where the notation $A^+$ indicates the Moore--Penrose pseudoinverse of a matrix $A$.

To solve \eqref{cmc_eq} with prescribed boundary via continuation, we let $\lambda$ be free, add a volume constraint to the problem, and consider the value of the volume, $V$, to be a bifurcation parameter. Then for a given CMC surface $\bx_0$, a new CMC surface $\bx = \bx_0 + \vphi \,\n_0$ is constructed after finding the solution to an equation of the form $f(\vphi,\lam,V) = 0$ which, in particular, fixes the mean curvature of the discrete surface. Iteratively more new solutions can be determined by setting $\bx_0 = \bx$ and repeating the process. Unlike many triangulation methods, this approach has no trouble constructing both stable and unstable solutions, and easily determines the locations of bifurcations, both simple and otherwise, which application-wise generally signify transitions in the qualitative behavior of the corresponding physical system, such as pinch-offs~\cite{eggers2008physics} or touchdowns~\cite{pelesko2003modeling}. Note that the stability to be considered is not with respect to arbitrary variations, but instead done with respect to those that are volume preserving, a choice that is related to whether pressure is applied directly or induced from hydrostatics. Such a distinction yields different results; see~\cite{bostwick2015stab} for details. However, by removing the corresponding volume constraint, the method can be easily modified to consider stability under arbitrary variations.

There are three other benefits of our method. First, we discretize the problem using a pseudospectral collocation~\cite{trefethen2000spectral}, so the method has high-order accuracy\footnote{Although, if high-order accuracy in not needed, finite differences can be used to take advantage of sparsity.}. Second, our approach is easily adaptable to compute field driven mean curvature surfaces that satisfy $2 H = f(\bx;\lam)$  and model interfaces deflected by a combination of magnetic~\cite{jamin2011instability}, electrostatic~\cite{brubaker2011nonlinear,moulton2009catenoid} and/or gravitational~\cite{bohme1980twodim,cohen2017shape} forces. Third, our method also provides a basis to extend to other geometric minimization problems that appear in biomechanics, geometry and soft-matter physics~\cite{cerda2003geometry,giomi2012minimal}.

In the next section we will provide the necessary theory of CMC surfaces. Afterwards, in \cref{sec:method}, we discuss how to place \eqref{cmc_eq} into the proper framework to apply numerical continuation, review pseudospectal discretizations and outline the proposed method. Then, in \cref{sec:disc_and_examp}, we present test results of reconstructed small and large cap solutions when the boundary of the CMC surfaces is the unit circle. Also, we present other examples from the literature that demonstrate the method's ability to capture bifurcations.

\section{Background on CMC ssurfaces}

Consider a two-dimensional surface $\Sigma$ immersed in $\R^3$ that is parameterized by the map $\bx \colon \Omega \to \R^3$ over the connected domain $\Omega \subset \R^2$. We will assume that the map $\bx$, with components $(x,y,z)$, is in $C^2(\Omega;\R^3)$, and is also regular at every point in $\Omega$. In denoting the coordinates of $\Omega$ as $(u,v)$, this restriction implies that the vectors $\bx_u \ce \partial_u \bx$ and $\bx_v \ce \partial_v \bx$ are linearly independent at every point in $\Omega$; hence, the cross product $\bx_u \times \bx_v$ does not vanish and the surface's normal vector
\[
 \n = \frac{\bx_u \times \bx_v}{|\bx_u \times \bx_v|}
\]
is well defined. The surface is then characterized by its metric tensor $g = \ip{\partial\tsp \bx}{\partial\tsp\bx}$ and the matrix of its second fundamental form $h = \ip{\partial^2\bx}{\n}$, whose components are given, respectively, by
\[ 
 \bgd
 g_{11} = E = \ip{\bx_{u}}{\bx_{u}}, \quad
 g_{12} = g_{21} = F = \ip{\bx_{u}}{\bx_{v}}, \quad
 g_{22} = G = \ip{\bx_{v}}{\bx_{v}}, \\
 h_{11} = L = \ip{\bx_{uu}}{\n}, \quad
 h_{12} = h_{21} = M = \ip{\bx_{uv}}{\n}, \quad
 h_{22} = N = \ip{\bx_{vv}}{\n}.
 \egd
\]
In local coordinates the shape operator (or Weingarten map) $S$ becomes $S = g^{-1} h$; hence,   the mean curvature is defined as $H = ({1}/{2})\trace{(S)}$, and the surface $\Sigma$ has CMC if and only if 
\be\tag{\ref{cmc_eq}}
 2H = \lam
\ee
for a given parameter $\lam$ in $\R$, which arises as a necessary and sufficient condition for determining the critical points of the surface area functional
\[
 \A[\bx] = \int_\Omega \sqrt{E G - F^2} \diff u \diff v
\]
over the set of $C^2$ regular surfaces that encapsulate a fixed volume $V$. The parameter $\lam$ in \eqref{cmc_eq} is a Lagrange multiplier used to enforce the volume constraint. 

With \eqref{cmc_eq} and a given volume $V$, finding CMC surfaces that span a given Jordan curve $\Gamma$ can then be framed as solving the following system for $(\bx,\lam)$ in $C^2(\Omega;\R^3) \times \R$:
\bsub\label{eq:cmc_prob}
\begin{gather}
 2H = \lam \quad \text{on }\Sigma, \quad 
 \bx|_{\partial \Omega} = \Gamma, \label{cmc_boundary}\\
 \V[\bx] \ce \frac{1}{3}\int_\Omega \bx \cdot (\bx_u \times \bx_v)  \diff u \diff v = V.\label{volume}
\end{gather}
\esub 
In \eqref{volume} the integral gives the signed volume of the region bounded by the surface and the cone connecting $\Gamma$ to the origin. Although not necessarily needed for finding CMC surfaces, \eqref{volume} is purposefully included---instead of fixing $\lam$, solving \eqref{cmc_boundary} and afterwards determining $V$---so that the volume of the liquid, $V$, can be used as a control parameter to quasistatically mimic evaporation or simulate the direct removal of a liquid~\cite{brubaker2015twodim_a}. 

\subsection{Stability}
Stability of a CMC surface $\bx$ is determined by the sign of the second variation of surface area over perturbations that preserve volume and fix the map's boundary data~\cite{bostwick2015stab}. The eigenvalue problem corresponding to the resulting bilinear form is 
\be\label{twist_eval}
 - \Delta_\Sigma \vphi - (\lam^2 - 2K)\vphi + \chi = \mu \vphi \quad  \text{in } \Omega, \quad
 \vphi = 0 \quad \text{on } \partial \Omega, \quad
 \int_\Sigma \vphi \diff \Sigma = 0
\ee
for $(\vphi,\lam) \in C^2(\Omega) \times \R$. Here $\vphi$ is the magnitude of the normal field displacement, $\Delta_\Sigma$ is the Laplace--Beltrami operator (or surface Laplacian) of $\Sigma$ defined via
\[
\Delta_\Sigma \vphi \equiv \frac{1}{\sqrt{|g|}}\div\big(\sqrt{|g|} g^{-1} \grad \vphi\big)
\]
and $K = \det(g^{-1}h) = (LN-M^2)/(EG-F^2)$ is the Gaussian curvature of the surface. The free parameter $\chi$ in \eqref{twist_eval} is set by the surface integral of $\vphi$, which is a condition ensuring that the disturbances are volume preserving. In the context of CMC surfaces, \eqref{twist_eval} is known as a twisted Dirichlet eigenvalue problem, which is self-adjoint and has a countable set of real eigenvalues, each of finite multiplicity~\cite{barbosa2000eigenvalue}. For the surface to be called stable, all of the eigenvalues must be greater than zero.

Aside from a few situations where geometric quantities simplify drastically~\cite{mccuan201extremities,lopez2012bifurcation,vogel1992stability}, the computation of the spectrum of \eqref{twist_eval} must be done numerically. As will be shown later, our method approximates the eigenvalues---and corresponding eigenmodes---using minimal extra work, which allows us to naturally determine when the index of the surface (i.e., the number of negative eigenvalues) changes as $V$ is varied.

\section{Method}\label{sec:method}

The basis of our method is to solve \eqref{eq:cmc_prob} using arc-length continuation, which will determine families of CMC surfaces that depend continuously on the parameters $\lam$ and $V$; however, as stated, \eqref{eq:cmc_prob} is underdetermined since the mean curvature, given locally in terms of the parameterization $\bx$, defines a map that takes $C^2(\Omega;\R^3)$ to $C(\Omega)$. To remove this indeterminacy a specific parameterization of the surface must be chosen. Inspired by theoretical results for bifurcations in CMC surfaces~\cite{koiso2002deformation}, we will look for solutions written as a normal graph over a given CMC surface. 

First, assume the surface $\bx_0 \colon \Omega \to \R^3$ (with normal $\n_0$) is a smooth immersion of constant mean curvature $\lam_0$ that satisfies $\bx_0|_{\partial \Omega} = \Gamma$ and has fixed signed volume, $\V[\bx_0] = V_0$. Then let $U$ be a sufficiently small open set of the Hölder space $C_0^{2,\alpha}(\Omega)$, containing the zero function $0$, chosen so the normal graph defined via $\bx = \bx_0 + \vphi \tsp\n_0$ is also an immersion for all $\vphi\in U$. Now denote the mean curvature of $\bx$ as $H(\vphi) \in C^{\alpha}(\Omega)$. With this setup, $\bx$ has constant mean curvature if and only if
\[
 2 H(\vphi) - \lam = 0
\] 
for some $\lam \in \R$. Observe that by assumption: (i) $ 2H(0) - \lam_0 = 0$; (ii) $\bx|_{\partial \Omega} = \Gamma$, since $\vphi$ vanishes on $\partial \Omega$; and (iii) $\V(0) = V_0$, where $\V(\vphi)$ denotes the signed volume of the immersion $\bx$. Thus, in defining the map $f \colon U \times \R \times \R \to C^\alpha(\Omega) \times \R$ by
\be\label{cont_funct}
 f(\vphi,\lam,V) = (2 H(\vphi) - \lam,\V(\vphi) - V),
\ee
the CMC problem \eqref{eq:cmc_prob} can be restated as 
solving the equation
\be\label{new_cmc_equation}
 f(\vphi,\lam,V) = (0,0).
\ee
Note that again $(\vphi,\lam,V) = (0,\lam_0,V_0)$ is indeed a solution of \eqref{new_cmc_equation}. 

Instead of determining the coordinates of a new surface directly, in this reformulation we search for a solution $\vphi \in C^{2,\alpha}_0(\Omega)$ of the nonlinear elliptic partial differential equation $2 H(\vphi) = \lam$, and then construct $\bx$, which is a surface of constant mean curvature $\lam$. Hence, the dimensionality of the system is reduced, and, upon discretizing, the resulting $f$ will provide a map between $\R^{k+1}$ and $\R^{k}$ for which arc-length continuation can be applied. We should remark that existence and uniqueness of solutions $(\vphi,\lam,V)$ of \eqref{new_cmc_equation} within a neighborhood of $\lam_0$ is guaranteed when the Jacobi operator $L \ce - \Delta_\Sigma  - (\ssp{\lam_0}{}^2 - 2K)$ over $H^1_0(\Sigma)$ has either no zero eigenvalues or the corresponding eigenspace is one dimensional with a basis vector of nonzero mean~\cite{koiso2002deformation}. (When these conditions fail, existence usually still holds but uniqueness fails due to a bifurcation.)

To construct a predictor step off the known solution $(0,\lam_0,V_0)$ for our continuation algorithm, we need to find a unique element in the null space of the Fréchet derivative of $f$. The following lemma proves necessary.
\begin{lemma}
The function $f$ defined in \eqref{cont_funct} is Fréchet differentiable with respect to $\vphi$, $\lam$ and $V$. Specifically, it can be shown that partial derivatives at $(\vphi,\lam,V) = (0,\lam_0,V_0)$ satisfy
\bsub\label{part_frechet_f}
\begin{gather}
 D_\vphi f(0,\lam_0,V_0) \psi = (\Delta \psi + (4{H}^2 - 2K)\psi, \int_\Omega \psi \sqrt{|g|} \diff \Omega) \quad \text{for }\ \psi \in C_0^2(\Omega), \\
 D_\lam f(0,\lam_0,V_0) \Lam = (-\Lam, 0), \qquad D_V f(0,\lam_0,V_0)W = (0, - W)
\end{gather}
\esub
where $H$, $K$ and $g$, respectively, are the mean curvature, Gauss curvature and metric tensor of the underlying CMC surface $\bx \colon \Omega \to \R^3$. As a result, the total derivative
\be\label{tot_frechet_f}
 f'(0,\lam_0,V_0)(\psi,\Lam,W) = (\Delta \psi + (4H^2 - 2K)\psi - \Lam, \int_\Omega \psi \sqrt{|g|} \diff \Omega - W)
\ee
for $(\psi,\Lam,W) \in C_0^2(\Omega) \times \R \times \R$.
\end{lemma}
\begin{proof}
Since $f$ is linear in $\lam$ and $V$, differentiability with respect to those parameters, along with their given formulas, is clearly valid. The result with respect to $\vphi$ follows from the Gâteaux derivative~(cf.~\cite{wente1980stability}), and expression~\eqref{tot_frechet_f} is produced from summing the partial derivatives in~\eqref{part_frechet_f}~\cite{chang2005na}.
\end{proof}

\subsection{Discretization}

For simplicity, let us assume that the reference domain $\Omega$ is given by the two-dimensional rectangle\footnote{More complicated reference domains can be constructed via, say, domain decomposition, although rectangular domain are often sufficient.} $(-l_u,l_u)\times(-l_v,l_v)$ with $l_u>0$ and $l_v >0$. Then $\Omega$ can then be discretized by the tensor-product grid $(u_i,v_j)$ of Chebyshev collocation points
\[
 (u_i,v_j) = (l_u \cos(i\pi/n), l_v \cos(j\pi/m)), \quad 
 i = 0,1, \ldots, n, \quad
 j = 0,1, \ldots, m, 
\]
and functions mapping $\Omega$ to $\R$ become $(n +1) \times (m + 1)$ matrices of values on this grid. To compute $u$- or $v$-derivatives of these functions, let $D_u$ be the standard one-dimensional Chebyshev differentiation matrix with scale factor $1/l_u$ acting on $n + 1$ nodes, and define $D_v$ similarly. Then $\partial_u$ and $\partial_v$ are approximated, respectively, by a right matrix multiplication of $\ssp{D_u}\tran$ and a left matrix multiplication of $D_v$, i.e., the discrete analogs of $x_u$ and $x_v$ are $X \ssp{D_u}\tran$ and $D_v X$, where $X$ is the matrix approximation of a function $x\colon\Omega \to \R$. Finally, using Kronecker products these expressions can be equated to $L_u\vx$ and $L_v\vx$, where $L_u = D_u \otimes I_v$ and $L_v = I_u \otimes D_v$ for the $(n + 1) \times (n + 1)$ and $(m + 1) \times (m + 1)$ identity matrices $I_u$ and $I_v$. Also, $\vx = \mathrm{vec}(X)$ or, in others words, is the vector resulting from stacking the columns of $X$. The second-order differentiation operators $\partial_{uu}$, $\partial_{uv}$ and $\partial_{vv}$ become $L_{uu} = \ssp{D_u}^2 \otimes I_v$, $L_{uv} = D_u \otimes D_v$ and $L_{vv} = I_u \otimes \ssp{D_v}^2$. 

Similarly, a matrix $\vect{w}\tran\in \R^{1 \times (n+1)(m+1)}$ acting on the vectorized functions that spectrally approximates integration over $\Omega$ can be defined via $\vect{w}\tran = \ssp{\vect{w}_u}\tran \otimes \ssp{\vect{w}_v}\tran$, where $\vect{w}_u$ and $\vect{w}_v$ are vectors of Clenshaw--Curtis quadrature weights for $n+1$ and $m+1$ nodes and with scale factors $l_u$ and $l_v$, respectively~\cite[p.\,126]{trefethen2000spectral}. Hence, the discrete version of $\int_\Omega x \diff u \diff v$ is the inner product $\vect{w}\tran \vx $.

\subsection{Algorithm}

With the above procedure, the initially known CMC surface $\bx_0 = (x_0, y_0,z_0)$ and normal vector $\n_0$ can be discretized coordinatewise as $\vect{X}_0 = [\vx_0; \vy_0; \vz_0]$ and $\vect{N}_0 = [\vn_{01}; \vn_{02}; \vn_{03}]$. The analogous discrete normal graph is $\vect{X} = \vect{X}_0 + \bm{\Phi} \circ \vect{N}_0$, where $\bm{\Phi} = [1;1;1] \otimes \bm{\vphi}$ for the vectorized discrete normal field function $\bm{\vphi}$ and $\circ$ denotes the Hadamard (entrywise) product. Accordingly, problem~\eqref{new_cmc_equation} transforms to the discrete system $f_\mathrm{d}(\bm{\vphi},\lam,V) = \vect{0}$, where $f_\mathrm{d}$ maps $\R^{k+2}$ to $\R^{k+1}$ for $k = (n+1)(m+1)$. The new function $f_\mathrm{d}$ is formed by calculating the necessary discrete differential geometric quantities of $\vect{X}$, including the coefficients of the fundamental forms, and then using pointwise and matrix multiplication to formulate discrete analogs of the expressions in \eqref{cont_funct}. A bordering strategy is used to apply the fixed boundary conditions, although resampling provides an intriguing option where more complicated boundary condition arise~\cite{driscoll2015rectangular}. By construction, $f_\mathrm{d}(\bm{0},\lam_0,V_0) = \vect{0}$ up to some specified tolerance. 

Given $(\bm{0},\lam_0,V_0)$, a new solution $(\bm{\vphi},\lam,V)$ of $f_\mathrm{d} = 0$ is found by first taking a predictor step in the direction $\vect{t}$ tangent to the bifurcation curve at $(\bm{0},\lam_0,V_0)$ via $(\bm{\psi},\Lam,W) = (\bm{0},\lam_0,V_0) + h \, \vect{t}$ and then applying the Newton iteration
\[
 \bgd
 (\bm{\psi}^{(i+1)},\lam^{(i+1)},V^{(i+1)}) \ce (\bm{\phi}^{(i)},\lam^{(i)},V^{(i)}) - J(\bm{\phi}^{(i)},\lam^{(i)},V^{(i)})^+ J(\bm{\phi}^{(i)},\lam^{(i)},V^{(i)}), \\
  (\bm{\phi}^{(0)},\lam^{(0)},V^{(0)}) \ce (\bm{\psi},\Lam,W).
 \egd
\] 
Recall $\vect{t}$ is the unique, normalized vector in the null space of matrix $J(\bm{0},\lam_0,V_0)$ in $\R^{(k+1)\times (k+2)}$, which is the Jacobian $J$ at $(\bm{0},\lam_0,V_0)$ of the discrete function $f_\mathrm{d}$. From this process, the new discrete CMC surface is $\vect{X} = \vect{X}_0 + \bm{\Phi} \circ \vect{N}_0$, and more discrete CMC surfaces can be computed by mapping $\vect{X}$ to $\vect{X}_0$ and reapplying the above two steps.

While the tangent vector $\vect{t}$ in the first step can be determined from the Jacobian $J$ of $f_\mathrm{d}$, we instead compute relevant objects from the continuous problem, then discretize and solve. Observe that the Fréchet derivative of $f$ is given in~\eqref{tot_frechet_f}, so the null space of $f'(0,\lam_0,V_0)$ is then set by  
\be\label{nullspace_prob}
 \Delta \psi + (\lam_0^2 - 2K)\psi = \Lam ,  \quad 
 \int_\Omega \psi |g|^{1/2} \diff \Omega = W.
\ee
for $(\psi,\Lam,W)$. By fixing $\Lam$ in $\R$, standard alternative theorems, coupled with CMC regularity results, prove that the partial differential equation in~\eqref{nullspace_prob} has a unique solution $\psi$ in $C^2_0(\Omega)$, when, again, zero is not an eigenvalue of the Jacobi operator~\cite[p.\,954]{lopez2012bifurcation}. With the function $\psi$, the value of $W$ can then be set from the integral of $\psi|g|^{1/2}$ over $\Omega$ and the resulting triple $(\psi,\Lam,W)$ produces the unique element, upon normalization, of the one-dimensional null space of \eqref{nullspace_prob} that can be discretized to produce $\vect{t}$. 

Instead of carrying out the above procedure to find $\vect{t}$, in practice we set $W = 1$, solve the discretized version of the equations
\be\label{null_prob}
 (-\Delta \psi - (\lam_0^2 - 2K)\psi + \Lam,\int_\Omega \psi |g|^{1/2} \diff \Omega) =  (0,1)
\ee
for $(\psi, \Lam)$, and renormalize the resulting discretized triple $(\psi,\Lam,1)$ using the Euclidean distance. The benefit of this approach is the left-hand side of \eqref{null_prob} defines the same linear operator for twisted eigenvalue problem~\eqref{twist_eval}, which determines the stability of the CMC surfaces. So with little extra work these eigenvalues $\mu_i$ can be approximated, along with the corresponding eigenmodes, from the discretized operator to produce the index of the initial surface CMC surface, i.e., the CMC surface $\bx_0$ corresponding to the solution $(0,\lam_0,V_0)$. This allows us to easily detect changes in stability and the onset of bifurcations by tracking the bifurcation test functional  
\[
 \beta(\bx_0,\lam_0,V_0) = \mathop{\mathrm{sign}}{\left(\prod_i \mu_i(\bx_0,\lam_0,V_0)\right)}.
\]

The discretized version of problem \eqref{null_prob} is find the unit vector $\vect{t} = (\vect{t}_\mathrm{p}, 1)/\alpha \in \R^{k+2}$ for $\vect{t}_\mathrm{p} \in \R^{k+1}$ such that%
\footnote{The subscript $\mathrm{p}$ denotes that the indicated object is only part of something larger, i.e., $J_p$ is only part of the Jacobian $J$.}
\be\label{discr_null_prob}
 {J}_\mathrm{p} \tsp \vect{t}_\mathrm{p} = \vect{e}_{k+2}, \quad
 {J}_\mathrm{p} \ce 
 \begin{bmatrix}
  -L & \vect{1}_{b} \\
  \vect{w}\tran\diag{(|\vect{g}|^{1/2})} & 0
 \end{bmatrix}
 \in \R^{(k+1)\times(k+1)},
\ee
where $\vect{e}_{k+2}$ is the last standard unit basis vector of $\R^{k+2}$.
The $k \times k$ real matrix $L$ satisfies 
\[
 L 
 = A_{0} L_{uu} + A_{1} L_{uv} +  A_{2} L_{vv} + A_{3} L_{u} + A_{4} L_{v}  
 + \diag(4\vect{H}^2 - 2\vect{K}) 
\] 
at nodes of the tensor product grid that correspond to the points on the interior of $\Omega$. In this definition, the $A_i$'s are the diagonal matrices achieved from discretizing and then vectorizing the coefficients
\[
 {G}/{|g|}, \ 
 - 2{F}/{|g|}, \
 {E}/{|g|}, \
 |g|^{-1/2}\div\big(|g|^{-1/2}{(G,-F)}\big), \
 |g|^{-1/2}\div\big(|g|^{-1/2}{(-F,E)}\big),
\]
respectively, and $\vect{H}$ and $\vect{K}$ are the discretized and vectorized mean and Gauss curvatures. At the other nodes, $L = I \in \R^{k\times k}$, which enforces the homogeneous Dirichlet boundary condition on $\vphi$. Also, the vector $\vect{1}_{b} \in \R^k$ is $0$ at the boundary nodes and 1 otherwise. Finally, $\vect{w}\tran\diag{(|\vect{g}|^{1/2})}$ is the discretized integral operator of the last term in expression~\eqref{tot_frechet_f}. 

Upon computing $J_\mathrm{d}$, the discretized version of the twisted eigenvalue problem~\eqref{twist_eval} for determining the stability of the $\bx_0$, which is approximated by $\vect{X}_0$, becomes the generalized eigenvalue problem  
\be\label{discr_evalue_prob}
 J_\mathrm{p}
 \vect{v}
 =
 \mu \,
 B
 \vect{v},
 \qquad
 B \ce 
 \begin{bmatrix}
 I_b &  0 \\
 \vect{0}\tran & 0 
 \end{bmatrix}
\ee
for $\vect{v} \in \R^{k+1}$, where $I_b$ is an identity matrix modified so that the rows corresponding to boundary nodes are set to zero.

The Jacobian $J$ for the corrector step can be formulated similarly as $J_\mathrm{p}$:
\be\label{discr_jacobian}
 J \ce 
 \begin{bmatrix}
  L & -\vect{1}_{b} & \vect{0} \\
  \vect{w}\tran\diag{(|\vect{g}|^{1/2})} & 0 & -1
 \end{bmatrix}
 \in \R^{(k+1)\times(k+2)}.
\ee
The components of this matrix are the same as those defined above, except for being constructed at the predicted point $(\bm{\psi},\Lam,W)$ instead of initial point $(\bm{0},\lam_0,V_0)$.

The following algorithm sketches the arc-length continuation method described, which incorporates Euler predictor and Newton corrector steps.

\bigskip
\hrule
\hrule
\smallskip
\noindent\textbf{Algorithm.}
\smallskip
\hrule

{\setstretch{0}
\begin{enumerate}[]
 \item \textbf{input} 
 \begin{enumerate}[\quad]
  \item \textbf{begin}
  \item $\vect{X}_0 = [\vx_0; \vy_0; \vz_0]$, $\lam_0$, $V_0$  \hfill \textit{CMC surface and parameter values}
  \item such that $f_\mathrm{d}(\vect{0},\lam_0,V_0) = \bm{0}$;
  \item $\bm{\vphi} = \vect{0}$; \hfill \textit{normal field}
  \item $h$; \hfill \textit{step length}
  \item \textbf{end}
 \end{enumerate} 
 \item \textbf{repeat} 
 \begin{enumerate}[\quad]
  \item construct $J_\mathrm{p}$ via \eqref{discr_null_prob} at $(\bm{0},\lam_0,V_0)$; \hfill \textit{unit tangent vector}
  \item find the eigenvalues of \eqref{discr_evalue_prob}; \hfill \textit{stability of $\vect{X}_0$}
  \item compute $\vect{t}$ via \eqref{discr_null_prob} and 
  \item $(\bm{\psi},\Lam,W) \ce (\bm{0},\lam_0,V_0) + h \tsp \vect{t}$;  \hfill \textit{predictor step}
  \item $\bm{\Psi} = [1;1;1] \otimes \bm{\vphi}, \quad \vect{X} = \vect{X}_0 + \bm{\Psi} \circ \vect{N}_0$;\hfill \textit{predicted surface}
  \item
  \item \textbf{repeat}
  \begin{enumerate}[\quad]   
  \item construct $J$ at $(\bm{\psi},\Lam,W)$ from \eqref{discr_jacobian};\hfill \textit{Jacobian}
  \item $(\til{\bm{\psi}},\til\Lam,\til W) \ce (\bm{\psi},\Lam,W) - J(\bm{\psi},\Lam,W)^+ J(\bm{\psi},\Lam,W);$ \hfill \textit{corrector loop}
  \item $(\bm{\psi},\Lam,W) \ce (\til{\bm{\psi}},\til\Lam,\til W);$ 
  \item $\bm{\Psi} = [1;1;1] \otimes \bm{\psi}, \quad \vect{X} = \vect{X}_0 + \bm{\Psi} \circ \vect{N}_0$; \hfill \textit{predicted surface}
  \end{enumerate}
  \textbf{until} convergence
  \item $\vect{X}_0 \ce \vect{X}, \ \lam_0 = \Lam, \ V \ce V_0$; \hfill \textit{new CMC surface}
 \end{enumerate}
\item \textbf{until} done traversing
\end{enumerate}

}

\smallskip
\hrule
\hrule
\medskip

\section{Examples}\label{sec:disc_and_examp}

\subsection{Spherical cap}

When the boundary of the surface is chosen to be the unit circle centered at the origin and contained in the horizontal plane $z = 0$, there is a continuous family of CMC spherical caps connected to the minimal planar disk. These surfaces can be uniquely characterized in terms of their enclosed volume $V$, which produces
\[
  \lam = 2H= -\frac{2 \pi^{1/3} \big(3V + \sqrt{\pi ^2 + 9V^2}-\pi ^{2/3} \big(\sqrt{\pi ^2 + 9V^2} - 3 V\big)^{1/3}\big)}{\sqrt{\pi ^2 + 9 V^2} \big(3V + \sqrt{\pi ^2 + 9 V^2}\big)^{1/3}}
\]
and a maximum vertical height
\[
 z_M = \frac{\left(3 V + \sqrt{9 V^2+\pi ^2}\right)^{2/3}-\pi ^{2/3}}{\pi^{1/3} (3V + \sqrt{9 V^2+\pi ^2})^{1/3}}.
\]
The resulting bifurcation diagram for positive signed volume is plotted in \cref{SphDropBifDia}.

Our algorithm was applied to construct these spherical caps starting from the flat interface $(x,y,0)$ with $(\lam,V) = (0,0)$; however, instead of using polar coordinates to initially discretize the horizontal coordinates---which introduces a coordinate singularity---the parameter domain $\Omega$ was chosen to be $(-1,1)^2$. Then $\Omega$ was mapped conformally~\cite{sctoolbox} into the unit disk, leaving the mean curvature unchanged, to produce a discretization of the flat surface. 

The algorithm was run for $m = n$ with $n = 12, 16, 20, 24, 32$ and $40$, yielding bifurcation diagrams that are nearly indistinguishable from the exact version shown in \cref{SphDropBifDia}(\textsc{left}). 
\begin{figure}[t!]
\begin{center}
 \includegraphics[width=1\textwidth]{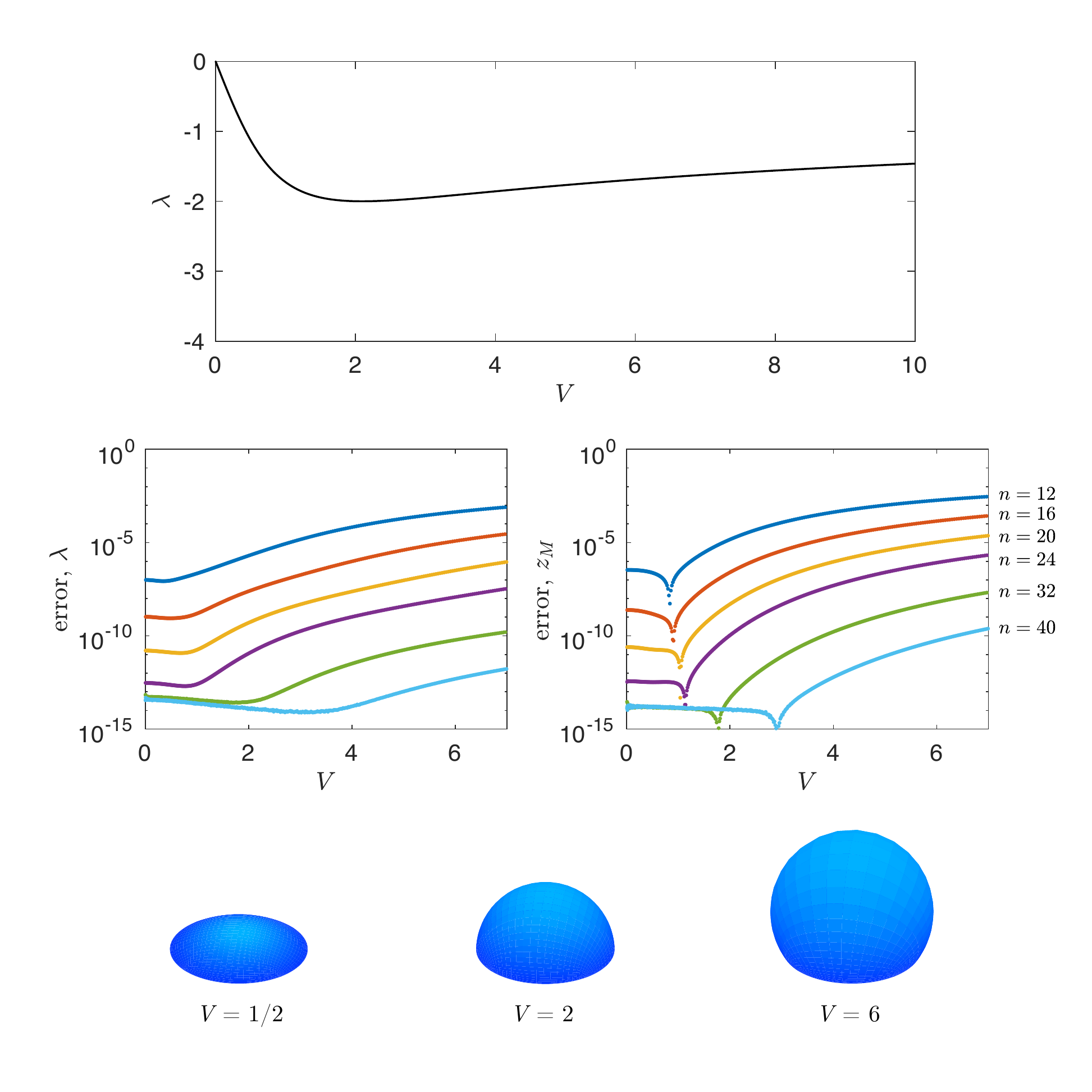}
\caption{\textsc{left:} Bifurcation diagram, plotting pressure $\lam = 2H$ versus volume $V$, of a CMC surface whose boundary is a unit circle. All of the solutions are stable, as determined by the sign of the first eigenvalue. \textsc{middle-left:} Relative error of $\lam$ versus volume $V$ for various discretizations with $m = n$. The value of $n$ for each curve follows the same sequence of discretizations shown in the middle-right panel . \textsc{middle-right:} Relative error of the maximum vertical height $z_M$ versus volume $V$ for various discretizations with $m = n$. \textsc{bottom:} Discrete spherical caps of volume $V = 1/2, 2$ and $6$ computed for $n = m = 40$.}
\label{SphDropBifDia}
\end{center}
\end{figure}
Given that the $V$ is prescribed precisely, the relative error between the exact and computed values of $\lam$ and $z_M$ are given in the top-right and bottom-right panels of \cref{SphDropBifDia}, respectively. These errors separately give measures on the accuracy of the mean curvature and the accuracy of the discrete representation of the surface. Observe that when $n = 24$ the algorithm predicts at least 7 digits of $\lam$ and 6 digits of $z_M$ correctly out to $V = 7$. Additionally, with $n = 24$ the average runtime of each predictor-correct step, including the computation of eigenvalue and eigenvectors, was a remarkable $0.14$ seconds on an iMac with a 1.8 GHz Intel Core i7 processor running MATLAB R2017a.

\subsection{Plateau--Rayleigh instability}

Cylinders $C$ of fixed length $2l$ and varying radius $r$ produce a family of CMC surfaces with $\lam = 1/r$. In 1873, Plateau~\cite{plateau1873stat} experimentally determined that, for decreasing radius, such cylinders become energetically unstable when $r = r_c \ce \alpha_c l $ for some constant critical constant $\alpha_c$, which was later determined by Rayleigh to be $(2\pi)^{-1}$~\cite{rayleigh1879instability}. This result, now know as the Plateau--Rayleigh instability, has subsequently been confirmed by numerous authors (e.g., see \cite{barbosa1984stability}) and extended to other similar situations~\cite{vogel1992stability,mccuan201extremities,lopez2012bifurcation}.

One extension is to consider a sessile drop attached to strip $\Omega = [-l,l] \times [-1,1]$ in the horizontal plane $\Pi_h = \R^2 \times \{0\}$ and bounded between two parallel vertical plates $\Pi_\pm = \{(x,y,z) \in \R^3 : y = \pm 1\}.$
Assuming that the tangent plane of the drop contacts $\Pi_\pm$ perpendicularly, the interface of the liquid takes on a cylindrical shape that can be parameterized by the mapping
\[
 \bx(u,v;r) = (r \sin(\zeta u/l),v,r\cos(\zeta u/l) - z_0)
\]
for $(u,v) \in \Omega = (-l,l) \times (-1,1)$. The parameters $r$, $\zeta$ and $z_0$ are defined as 
\[
 r = \frac{l^2+t^2}{2 t}, \quad 
 \zeta = \cos ^{-1}\left(\frac{l^2-t^2}{l^2+t^2}\right), \quad
 z_0 =\frac{l^2-t^2}{2 t},
\]
in terms of the free parameter $t \in (0,\infty)$ that represents the maximum vertical height of the drop. The pressure $\lam$ and  volume $V$ are consequently given by 
\[
 \lam = -1/r, \quad V = 2(r^2\zeta - l z_0), 
\]
and the resulting bifurcation diagram is shown in \cref{LiqRivulet_fig}. Note that $V:\R^+ \to \R^+$ is a bijective function of $t$ and, thus, can be used as a free parameter instead of $t$.

Stability of these cylindrical surfaces is determined by a variant of problem~\eqref{twist_eval} where the Dirichlet conditions on lines $v = \pm 1$ are replaced by homogeneous Neumann conditions since the tangent space of the drop on those boundaries contacts $\Pi_\pm$ perpendicularly. The resulting minimal eigenvalue is 
\[
 \mu_0 = \lam^2 \left(\frac{\pi^2}{4 \zeta^2}-1\right)+\frac{\pi^2}{4};
\]
see~\cite{mccuan201extremities}. When $V = 0$, the minimal eigenvalue $\mu_0$ is positive. Then as the volume increases, $\mu_0$ monotonically decreases to a negative global minimum. If $l < 0.3647 \ldots$, then $\min_{V \in(0,\infty)} \mu_0(V) < 0$ and  stability is lost (i.e., $\mu_0 = 0$) for $V = V^*$ at a super-critical pitchfork bifurcation, where two translationally asymmetric surfaces appear. In particular, when $l = 1/5$, the critical value $V^* = 0.140135\ldots$.

\cref{LiqRivulet_fig}(\textsc{top}) shows the bifurcation diagram for $l = 1/5$ computed with $n = 32$. Initially, the branch emanating from the point $(V,2H) = (0,0)$ was calculated, along with the minimal eigenvalue of each solution. A change in stability is predicted at $V = 0.140135$, a value with at least seven digits of accuracy. The branch of asymmetric solutions bifurcating from this point was traced by using the appropriate eigenvector as the predictor direction instead of the vector specified from the null space problem. Although these solutions are hard to reconstruct analytically, our proposed method has no trouble computing them. Finally, given $V$ exactly, the relative errors in the parameter $\lam$ and the minimal eigenvalue $\mu_0$ on the cylindrical-solution branch are plotted in the bottom two panels of \cref{LiqBrdgeFig}.

\begin{figure}[htbp!]
\begin{center}
 \includegraphics[width=1\textwidth]{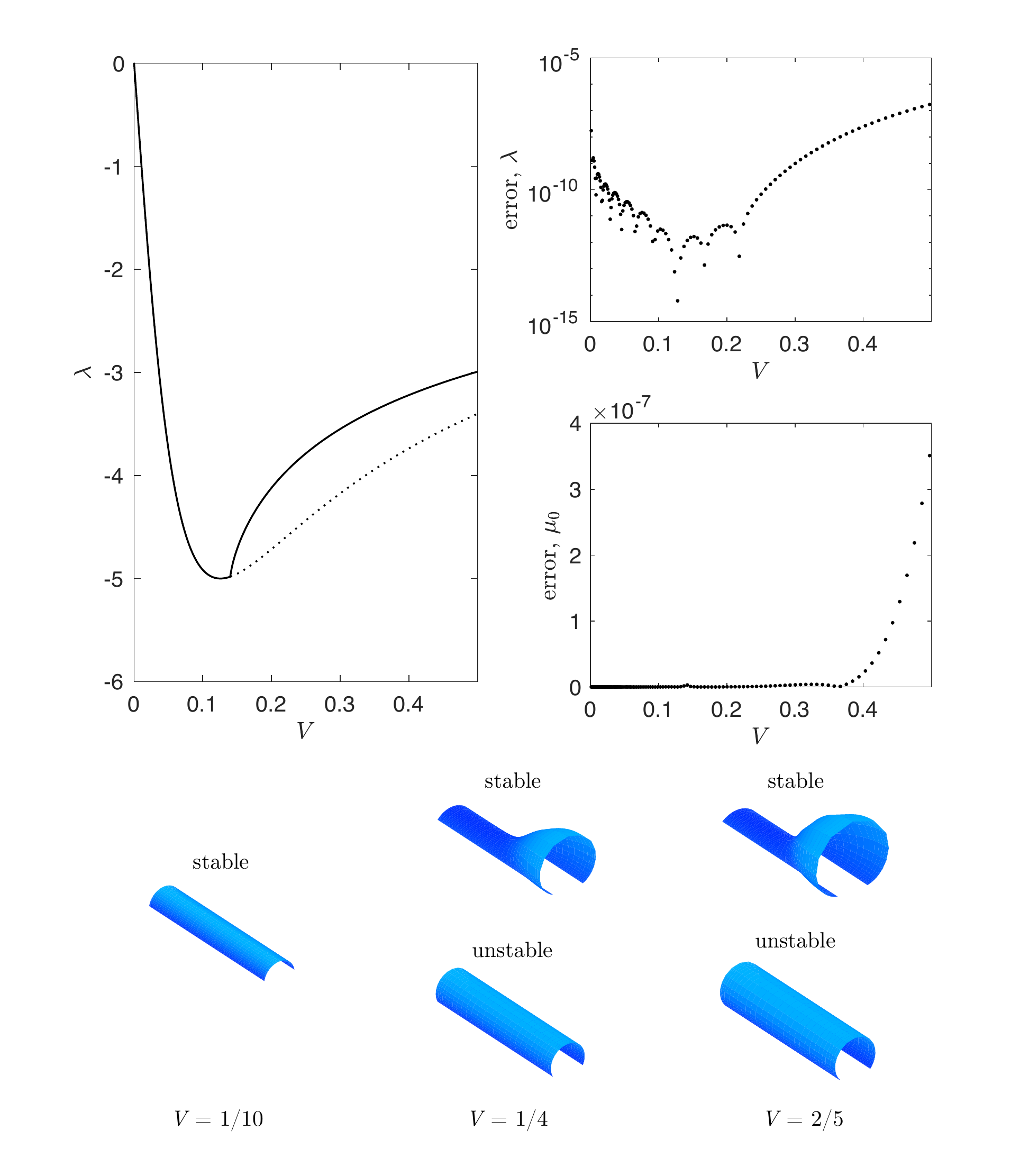}
\caption{\textsc{top:} Bifurcation diagram  of CMC surfaces in the rectangular strip $[-l,l]\times [-1,1]$ with $l = 1/5$ (calculated with $n = m = 32$) plotting the pressure $\lam$ versus the volume $V$ enclosed by the surface. Solid and dashed lines, respectively, represent stable and unstable solutions. The branch connected to $(V,\lam) = (0,0)$, consisting of translationally symmetric cylindrical surfaces, is stable up to $V = 0.140135\ldots$, at which point stability is lost and two asymmetric stable CMC surfaces appear; see the bottom panel for example surfaces. \textsc{top-right:} Uniform norm error in the calculated $\lam$ of the cylindrical-solution branch of the numerically computed bifurcation diagram. \textsc{middle-right:} Error of the numerically calculated first eigenvalue $\mu_0$, again, of the cylindrical-solution branch of the numerically computed bifurcation diagram. \textsc{bottom:} Computed CMC surfaces having volume $V = 1/10, 1/4 $ and $2/5$, with their corresponding stability labeled.}
\label{LiqRivulet_fig}
\end{center}
\end{figure}

\subsection{Liquid bridges}

Analytically predicting stability changes and bifurcations in families of CMC surfaces is a difficult task, even for many simple boundaries. Explicit calculations of the transitions are often limited to situations where the underlying geometry is known, like in the cylinder example above, or where special techniques can be used. 

For example, consider a volume of fluid stretched between two coaxial unit rings in parallel planes separated by a distance $1$. Such a configuration is known as a liquid bridge, and if the fluid shape is radially symmetric, the interface is given by a section of a Delaunay surface~\cite{kenmotsu2003surfaces}. Stability of symmetric liquid bridges was established in~\cite{patnaik1994volume} by investigating the behavior of a non-constant coefficient Sturm--Liouville problem derived from separating variables in~\eqref{twist_eval}. When the volume $V$ is sufficiently small, the rotationally symmetric solutions are physically stable; however, if $V$ is increased the past critical value (determined when the tangent plane on the boundary of the surface becomes parallel with the rings' plane) such  solutions become unstable and a branch of rotationally asymmetric surfaces with a radial bulge emerges~\cite{russo1986instability}. Numerical constructions of the nonrotational surfaces appeared in~\cite{hoffman1993comment}.

Starting from an initial cylinder, our method easily finds both the symmetric and asymmetric branches, along with their stability and the critical value at which the bifurcation occurs. The resulting diagram, with representative surfaces, is shown in \cref{LiqBrdgeFig}. 
\begin{figure}[htbp!]
\begin{center}
 \includegraphics[width=1\textwidth]{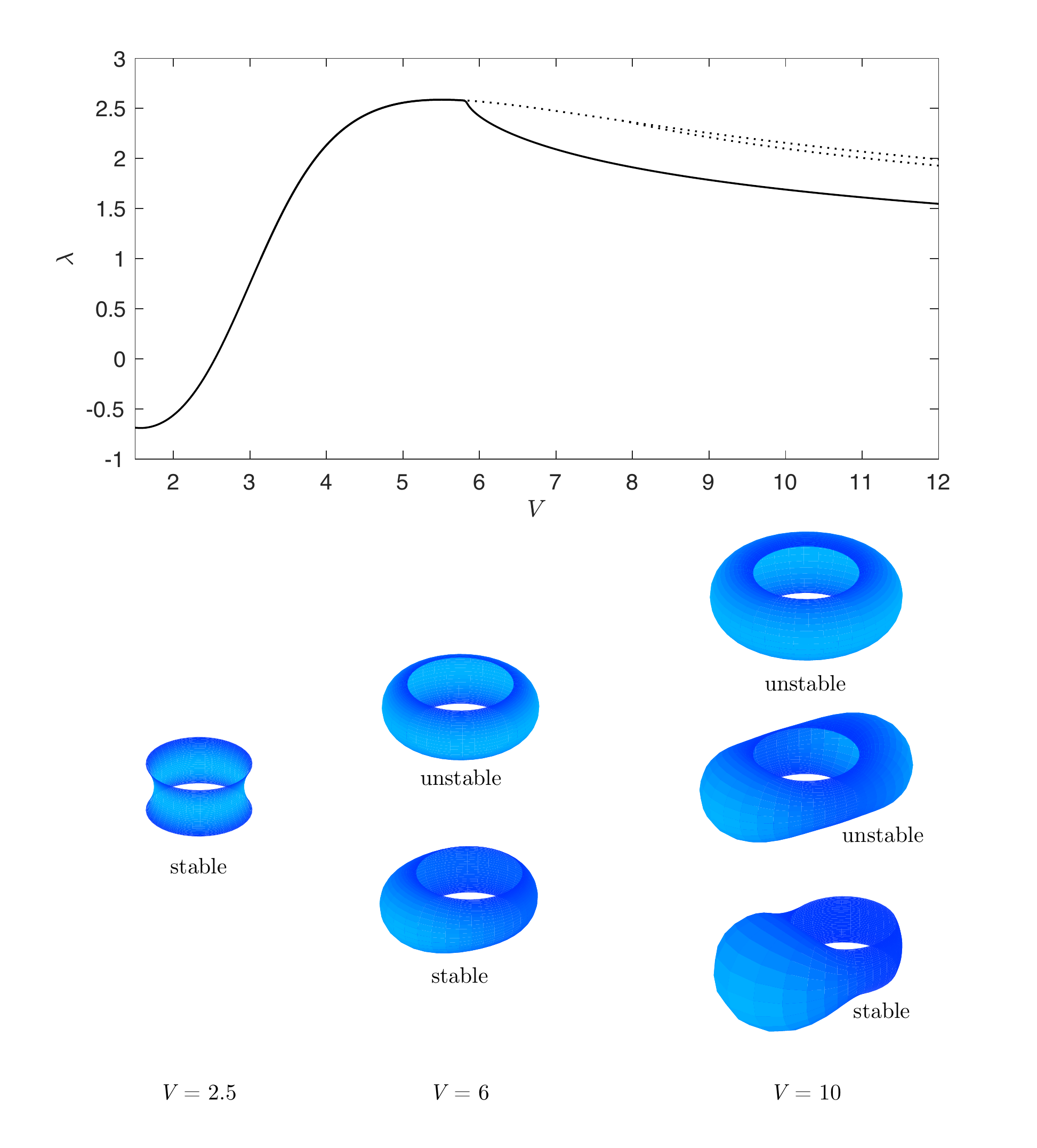}
\caption{\textsc{top:} Bifurcation diagram, plotting pressure $\lam$ versus volume $V$, of  constant mean curvature liquid bridges spanning parallel coaxial rings, of radius 1, separated by a distance $h = 1$. Again, solid and dashed lines, respectively, represent stable and unstable surfaces. The branch existing for $V < 3$ represents rotationally symmetric solutions. These solutions are stable for smaller volumes but become unstable at critical volume near $V = 6$, where stable family rotationally-invariant and asymmetric one-bulge solutions arise. Additionally, another symmetry breaking bifurcation occurs near $V = 8$: a family of rotationally-invariant surfaces with two radial bulges branches from of the axially symmetric surface. However, these solutions remain unstable, instead being stable and forming a sub-critcal pitchfork bifurcation. \textsc{bottom:} Computed CMC surfaces with the volume and stability labeled.}
\label{LiqBrdgeFig}
\end{center}
\end{figure}
Observe that the method also finds a third family of solutions, which bifurcates off the branch of symmetric solutions, that consists of nonrotational surfaces with two bulges. These two-mode solutions are unstable and, thus, cannot be constructed using direct energy minimization methods such as Surface Evolver.

\section{Conclusions}\label{sec:conc}

We have given a new approach for calculating discrete CMC surfaces that is an alternative to the standard algorithms involving surface triangulations or integrable systems theory. The method uses a normal graph parameterization to transform the standard CMC equation to a quasilinear second-order partial differential equation, which reduces the dimensionality of and removes invariances in unknown surface coordinates. Afterwards, the partial differential equation is solved with numerical arc-length continuation by shooting on the volume enclosed by the surface. 

The examples presented demonstrate that the procedure is fast, often less than a second per predictor-corrector iteration; is highly accurate, since a pseudospectral method is used for discretization; and easily constructs families of CMC surfaces with fixed boundary from a known CMC surface. Also, the algorithm identifies the stability of solutions and accurately detects bifurcation points.

A limitation of the algorithm as presented is it is relies on an underlying rectangular parameterization domain $\Omega$; however, this difficulty can be circumvented with other minor adaptations. For example, domain decomposition can be applied in a standard fashion, or invariances can be capitalized upon to conformally map $\Omega$ to an appropriate flat initial surface in $\R^3$; see section~\ref{sec:disc_and_examp}. Additionally, substituting in finite elements for pseudospectral collocation allows for added requisite, although high-order accuracy is sacrificed.


\begin{thebibliography}{10}
\footnotesize{


\bibitem{allgower1990numerical}
{\sc E.~L. Allgower and K.~Georg}, {\em Numerical Continuation Methods: An
  Introduction}, vol.~13 of Springer Ser. Comput. Math., Springer-Verlag,
  Berlin, 1990.

\bibitem{barbosa1984stability}
{\sc J.~L. Barbosa and M.~do~Carmo}, {\em Stability of hypersurfaces with
  constant mean curvature}, Math. Z., 185 (1984), pp.~339--353.

\bibitem{barbosa2000eigenvalue}
{\sc L.~Barbosa and P.~B{\'e}rard}, {\em Eigenvalue and ``twisted'' eigenvalue
  problems, applications to {CMC} surfaces}, J. Math. Pures Appl. (9), 79
  (2000), pp.~427--450.

\bibitem{bobenko1999disc}
{\sc A.~I. Bobenko and U.~Pinkall}, {\em Discretization of surfaces and
  integrable systems}, in Discrete Integrable Geometry and Physics, Oxford
  Lecture Ser. Math. Appl., Oxford Univ. Press, Oxford, 1999, pp.~3--58.

\bibitem{bohme1980twodim}
{\sc R.~B{\"o}hme, S.~Hildebrandt, and E.~Tausch}, {\em The two-dimensional
  analogue of the catenary}, Pacific J. Math., 88 (1980), pp.~247--278.

\bibitem{bostwick2015stab}
{\sc J.~B. Bostwick and P.~H. Steen}, {\em Stability of constrained capillary
  surfaces}, Annu. Rev. Fluid Mech., 47 (2015), pp.~539--568.

\bibitem{brakke1992surface}
{\sc K.~Brakke}, {\em The surface evolver}, Exp. Math., 1 (1992), pp.~141--165.

\bibitem{brubaker2015twodim_a}
{\sc N.~D. Brubaker and J.~Lega}, {\em Two-dimensional capillary origami with
  pinned contact line}, SIAM J. Appl. Math., 75 (2015), pp.~1275--1300.

\bibitem{brubaker2011nonlinear}
{\sc N.~D. Brubaker and J.~A. Pelesko}, {\em Non-linear effects on canonical
  {MEMS} models}, European J. Appl. Math., 22 (2011), pp.~455--470.

\bibitem{cerda2003geometry}
{\sc E.~Cerda and L.~Mahadevan}, {\em Geometry and physics of wrinkling}, Phys.
  Rev. Lett., 90 (2003), p.~074302.

\bibitem{chang2005na}
{\sc K.-C. Chang}, {\em Methods in Nonlinear Analysis}, Springer Monogr. Math.,
  Springer, Berlin, 2005.

\bibitem{cohen2017shape}
{\sc C.~Cohen, B.~Darbois~Texier, E.~Reyssat, J.~H. Snoeijer, D.~Qu{\'e}r{\'e},
  and C.~Clanet}, {\em On the shape of giant soap bubbles}, Proc. Natl. Acad.
  Sci. USA, 114 (2017), pp.~2515--2519.

\bibitem{collicott2004computing}
{\sc S.~Collicott and M.~Weislogel}, {\em Computing existence and stability of
  capillary surfaces using surface evolver}, AIAA J., 42 (2004), pp.~289--295.

\bibitem{crane2013robust}
{\sc K.~Crane, U.~Pinkall, and P.~Schr{\"o}der}, {\em Robust fairing via
  conformal curvature flow}, ACM Trans. Graph., 32 (2013).

\bibitem{doedel1991numerical}
{\sc E.~Doedel, H.~B. Keller, and J.~P. Kernevez}, {\em Numerical analysis and
  control of bifurcation problems {(II)}: {B}ifurcation in infinite
  dimensions}, Int. J. Bifur. Chaos, 1 (1991), pp.~745--772.

\bibitem{doedel2007auto}
{\sc E.~J. Doedel, A.~R. Champneys, F.~Dercole, T.~F. Fairgrieve, A.~Yu,
  B.~Oldeman, R.~Paffenroth, B.~Sandstede, X.~J. Wang, and C.~H. Zhang}, {\em
  {AUTO-07P}: Continuation and Bifurcation Software for Ordinary Differential
  Equations}, 2007.

\bibitem{sctoolbox}
{\sc T.~A. Driscoll}, {\em {Schwarz--Christoffel Toolbox for MATLAB}},
  \url{http://www.math.udel.edu/~driscoll/SC/}.

\bibitem{driscoll2015rectangular}
{\sc T.~A. Driscoll and N.~Hale}, {\em Rectangular spectral collocation}, IMA
  J. Numer. Anal., 36 (2015), pp.~108--132.

\bibitem{dzuik2006fem}
{\sc G.~Dziuk and J.~E. Hutchinson}, {\em Finite element approximations to
  surfaces of prescribed variable mean curvature}, Numer. Math., 102 (2006),
  pp.~611--648.

\bibitem{eggers2008physics}
{\sc J.~Eggers and E.~Villermaux}, {\em Physics of liquid jets}, Rep. Progr.
  Phys., 71 (2008), p.~036601.

\bibitem{finn1986equilibrium}
{\sc R.~Finn}, {\em Equilibrium capillary surfaces}, Grundlehren Math. Wiss
  284, Springer, New York, 1986.

\bibitem{finn1992greens}
{\sc R.~Finn}, {\em Green's identities and pendent liquid drops, {I}}, in
  Developments in Partial Differential Equations and Applications to
  Mathematical Physics, G.~Buttazzo, G.~P. Galdi, and L.~Zanghirati, eds.,
  Springer, Boston, 1992, pp.~39--58.

\bibitem{finn2002eight}
{\sc R.~Finn}, {\em Eight remarkable properties of capillary surfaces}, Math.
  Intelligencer, 24 (2002), pp.~21--33.

\bibitem{lipowsky1999liquid}
{\sc H.~Gau, S.~Herminghaus, P.~Lenz, and R.~Lipowsky}, {\em Liquid
  microchannels on structured surfaces}, Science, 283 (1999), pp.~46--49.

\bibitem{giomi2012minimal}
{\sc L.~Giomi and L.~Mahadevan}, {\em Minimal surfaces bounded by elastic
  lines}, R. Soc. Lond. Proc. Ser. A Math. Phys. Eng. Sci., 468 (2012),
  pp.~1851--1864.

\bibitem{ehinz1969nonexist}
{\sc E.~Heinz}, {\em On the nonexistence of a surface of constant mean
  curvature with finite area and prescribed rectifiable boundary}, Arch.
  Ration. Mech. Anal., 35 (1969), pp.~249--252.

\bibitem{hoffman1993comment}
{\sc D.~Hoffman}, {\em Comment utiliser un ordinateur pour trouver de nouvelles
  surfaces minimales et des bulles de savon}, in Surfaces minimales et
  solutions de probl{\`e}mes variationnels, SMF Journ. Annu., Soci{\'e}t{\'e}
  Math{\'e}matique de France, Paris, 1993, pp.~29, 10--11, 270--271.

\bibitem{dorfmeister1998weierstrass}
{\sc F.~P. J.~Dorfmeister and H.~Wu}, {\em Weierstrass type representation of
  harmonic maps into symmetric spaces}, Comm. Anal. Geom., 6 (1998),
  pp.~633--668.

\bibitem{jamin2011instability}
{\sc T.~Jamin, C.~Py, and E.~Falcon}, {\em Instability of the origami of a
  ferrofluid drop in a magnetic field}, Phys. Rev. Lett., 107 (2011),
  p.~204503.

\bibitem{kapouleas1991compact}
{\sc N.~Kapouleas}, {\em Compact constant mean curvature surfaces in euclidean
  three-space}, J. Differential Geom., 33 (1991), pp.~683--715.

\bibitem{keller1987lectures}
{\sc H.~B. Keller}, {\em Lectures on Numerical Methods In Bifurcation
  Problems}, Tata Inst. Lectures Math. 79, Springer-Verlag, Berlin, 1987.

\bibitem{kenmotsu2003surfaces}
{\sc K.~Kenmotsu}, {\em Surfaces with Constant Mean Curvature}, vol.~221 of
  Transl. Math. Monogr, Amer. Math. Soc., Providence, RI, 2003.

\bibitem{koiso2002deformation}
{\sc M.~Koiso}, {\em Deformation and stability of surfaces with constant mean
  curvature}, Tohoku Math. J. (2), 54 (2002), pp.~145--159.

\bibitem{lau2003super}
{\sc K.~K.~S. Lau, J.~Bico, K.~B.~K. Teo, M.~Chhowalla, G.~A.~J. Amaratunga,
  W.~I. Milne, G.~H. McKinley, , and K.~K. Gleason}, {\em Superhydrophobic
  carbon nanotube forests}, Nano Lett., 3 (2003), pp.~1701--1705.

\bibitem{lobaton2007computation}
{\sc E.~J. Lobaton and T.~R. Salamon}, {\em Computation of constant mean
  curvature surfaces: Application to the gas--liquid interface of a pressurized
  fluid on a superhydrophobic surface}, J. Colloid Interface Sci., 314 (2007),
  pp.~184--198.

\bibitem{lopez2012bifurcation}
{\sc R.~L{\'o}pez}, {\em Bifurcation of cylinders for wetting and dewetting
  models with striped geometry}, SIAM J. Math. Anal., 44 (2012), pp.~946--965.

\bibitem{lopez2013cmc}
{\sc R.~L{\'o}pez}, {\em Constant Mean Curvature Surfaces with Boundary},
  Springer Monogr. Math., Spinger, 2013.

\bibitem{mccuan201extremities}
{\sc J.~McCuan}, {\em Extremities of stability for pendant drops}, in Geometric
  Analysis, Mathematical Relativity, and Nonlinear Partial Differential
  Equations, Contemp. Math. 599, American Mathematical Society, Providence, RI,
  2013.

\bibitem{metzger2004numerical}
{\sc J.~Metzger}, {\em Numerical computation of constant mean curvature
  surfaces using finite elements}, Classical Quantum Grav., 21 (2004), p.~4625.

\bibitem{moulton2009catenoid}
{\sc D.~E. Moulton and J.~A. Pelesko}, {\em Catenoid in an electric field},
  SIAM J. Appl. Math., 70 (2009), pp.~212--230.

\bibitem{myshkis1987lowgrav}
{\sc A.~Myshkis, V.~B. N. K. L. S.~A. Myshkis, A.~Tyuptsov, N.~Kopachevskii,
  L.~Slobozhanin, and A.~Tyuptsov}, {\em Low-Gravity Fluid Mechanics:
  Mathematical Theory of Capillary Phenomena}, Springer, New York, 1987.

\bibitem{pan2012robust}
{\sc H.~Pan, Y.-K. Choi, Y.~Liu, W.~Hu, Q.~Du, K.~Polthier, C.~Zhang, and
  W.~Wang}, {\em Robust modeling of constant mean curvature surfaces}, ACM
  Trans. Graph., 31 (2012), p.~85.

\bibitem{patnaik1994volume}
{\sc U.~Patnaik}, {\em Volume constrained {D}ouglas problem and the stability
  of liquid bridges between two coaxial tubes}, PhD thesis, The University of
  Toledo, Toledo, OH, 1994.

\bibitem{paulsen2015optimal}
{\sc J.~D. Paulsen, V.~Demery, C.~D. Santangelo, T.~P. Russell, B.~Davidovitch,
  and N.~Menon}, {\em Optimal wrapping of liquid droplets with ultrathin
  sheets}, Nature Mater., 14 (2015), pp.~1206--1209.

\bibitem{pelesko2003modeling}
{\sc J.~A. Pelesko and D.~H. Bernstein}, {\em Modeling {MEMS} and {NEMS}},
  Chapman \& Hall/CRC, Boca Raton, FL, 2003.

\bibitem{peraud2014geometry}
{\sc J.-P. P\'eraud and E.~Lauga}, {\em Geometry and wetting of capillary
  folding}, Phys. Rev. E, 89 (2014), p.~043011.

\bibitem{plateau1873stat}
{\sc J.~A.~F. Plateau}, {\em Statique exp{\'e}rimentale et th{\'e}oretique des
  liquides soumis aux seule forces mol{\'e}culaires}, Gautier-Villars, Paris,
  1873.

\bibitem{polthier2002discrete}
{\sc K.~Polthier and W.~Rossman}, {\em Discrete constant mean curvature
  surfaces and their index}, J. Reine Angew. Math, 549 (2002), pp.~47--77.

\bibitem{rayleigh1879instability}
{\sc J.~W.~S. Rayleigh}, {\em On the instability of jets}, Proc. London Math.
  Soc., s1-10 (1878), pp.~4--13.

\bibitem{renka2015simple}
{\sc R.~J. Renka}, {\em A simple and efficient method for modeling constant
  mean curvature surfaces}, SIAM J. Sci. Comput., 37 (2015), pp.~A2076--A2099.

\bibitem{russo1986instability}
{\sc M.~J. Russo and P.~H. Steen}, {\em Instability of rotund capillary bridges
  to general disturbances: Experiment and theory}, J. Colloid Interface Sci.,
  113 (1986), pp.~154--163.

\bibitem{singler1996computer}
{\sc T.~J. Singler, X.~Zhang, and K.~A. Brakke}, {\em Computer simulation of
  solder bridging phenomena}, J. Electron. Packag, 118 (1996), pp.~122--126.

\bibitem{tas1996stiction}
{\sc N.~R. Tas, T.~Sonnenberg, H.~Jansen, R.~Legtenberg, and M.~Elwenspoek},
  {\em Stiction in surface micromachining}, J. Micromech. Microeng., 6 (1996),
  pp.~385--397.

\bibitem{trefethen2000spectral}
{\sc L.~N. Trefethen}, {\em Spectral methods in {MATLAB}}, Software, Environ.
  Tools, SIAM, Philadelphia, 2000.

\bibitem{vogel1992stability}
{\sc T.~I. Vogel}, {\em Stability and bifurcation of a surface of constant mean
  curvature in a wedge}, Indiana Univ. Math. J., 41 (1992), pp.~625--648.

\bibitem{wente1980stability}
{\sc H.~C. Wente}, {\em The stability of the axially symmetric pendent drop},
  Pacific J. Math., 88 (1980), pp.~421--470.

\bibitem{wente1986counterexample}
{\sc H.~C. Wente}, {\em Counterexample to a conjecture of {H. Hopf}}, Pacific
  J. Math., 121 (1986), pp.~193--243.
  
  }

\end{thebibliography}
\end{document}